\documentclass[12pt]{amsart}

\makeatletter
\@namedef{subjclassname@2020}{%
  \textup{2020} Mathematics Subject Classification}
\makeatother 

\usepackage{latexsym}
\usepackage{amsmath,amsthm,amssymb}
\usepackage{mathptmx}

\usepackage[dvips]{graphicx,color}

\newtheorem{df}{Definition}[section]
\newtheorem{thm}[df]{Theorem}
\newtheorem{cor}[df]{Corollary}
\newtheorem{prop}[df]{Proposition}
\newtheorem{lem}[df]{Lemma}

\newtheorem{rem}[df]{Remark}

\newcommand{\Cl}[1]{\operatorname{Cl}_{#1}}
\newcommand{\Int}[1]{\operatorname{Int}_{#1}}
\newcommand{\Fr}[1]{\operatorname{Fr}_{#1}}
\newcommand{\Ext}[1]{\operatorname{Ext}_{#1}}
\newcommand{\Diam}[1]{\operatorname{diam}_{#1}}
\newcommand{\pois}{\smallsetminus}
\newcommand{\St}[2]{\operatorname{St}({#1},\mathcal{{#2}})}

\newcommand{\mapright}[1]{%
\smash{\mathop{%
\hbox to 1cm{\rightarrowfill}}\limits^{#1}}}

\begin{document}
 
\sloppy

\title
{Connected properties of the sublinear Higson corona of $\Bbb R^n$}
 
\author{Yuji Akaike}

\date{\today}

\address{National Institute of Technology, Kure College,
2-2-11 Aga-Minami Kure-shi Hiroshima 737-8506, Japan}
 
\email{akaike@kure-nct.ac.jp}

\keywords{mutual aposyndesis, sublinear Higson compactification.}

\subjclass[2020]{Primary 54D35, 54D40; Secondary 54D05.}

\begin{abstract}
In \cite{AT}, we showed that the sublinear Higson corona
 of the half open interval $[0,\infty)$ with the usual metric is an indecomposable continuum.
In this paper, for the $n$-dimensional Euclidean space $\Bbb R^n$ with the usual metric and $n\geq2$,
 we show that the sublinear Higson corona of $\Bbb R^n$ is not locally connected at any point
 and is mutually aposyndetic.
\end{abstract}

\maketitle 

\section{Introduction}

For a topological space $X$ and a subspace $Y\subset X$,
the closure, the interior, and the boundary is denoted by $\Cl{X}Y$, $\Int{X}Y$, and $\Fr{X}Y$, respectively.
For a metric space $(X,d)$, let 
$B_d(x,r)=\{y\in X~|~d(x,y)<r\}$
 and $N_d(x,r)=\{y\in X~|~d(x,y)\leq r\}$ for any $x\in X$ and $r>0$.
A metric space $(X,d)$ is called {\it proper} if 
$N_d(x,r)$ is compact for all $x\in X$ and $r>0$.
In this paper, all metric spaces are assumed to be noncompact proper
 with a base point $x_0\in X$, and $|x|=d(x,x_0)$ for each $x\in X$.
For a subset $A$ of a metric space $(X,d)$, the diameter of $A$ is denoted by $\Diam{d}A$.
Unless explicitly stated to the contrary, we assume that all maps are continuous.

Recall that the notions of Higson compactification and its remainder, Higson corona,
 play an important role in large-scale geometry \cite{Roe}.
Later, the sublinear Higson compactification, the subpower Higson compactification 
 and these coronas are introduced 
 and investigated in \cite{CDSV}, \cite{DS}, \cite{KZ1} as a counterpart of the Higson compactification.
Note that these compactifications are metric-dependent.
In \cite{AT}, perfectness of the sublinear Higson compactification is characterized and
 is proved that the sublinear Higson corona of the half open interval $[0,\infty)$ with the usual metric
 is an indecomposable continuum.

On the other hand,
The concept of aposyndesis is introduced in \cite{Jones}.
Aposyndesis and mutual aposyndesis are analogues of
 separation axioms $T_1$ and $T_2$, respectively,
 which are used to understand differences of continua. 
It is known that every locally connected continuum is aposyndetic and
 every aposyndetic continuum is decomposable (cf. \cite{Ma}).
Let $\Bbb R^n$ be the $n$-dimensional Euclidean space. 
In \cite{Bella}, Bellamy showed that, for each $n\geq2$, 
 the Stone-\v{C}ech remainder $\beta\Bbb R^n\pois\Bbb R^n$ is aposyndetic,
 and five decades later, he proved in \cite{Bella2} that $\beta\Bbb R^n\pois\Bbb R^n$
 is mutually aposyndetic for each $n\geq2$.
In \cite{Woods}, Woods proved that if $X$ is realcompact and noncompact,
 then $\beta X\pois X$ is not connected im kleinen at any point.
Hence $\beta\Bbb R^n\pois\Bbb R^n$ is not locally connected at any point.
In \cite{Kees}, Keesling proved, using algebraic methods,
 that the Higson corona of $\Bbb R^n$ with the usual metric is not locally connected at any point.

In this paper, 
 for any proper metric space $(X,d)$,
 we show that the sublinear Higson corona of $X$ is
 not locally connected at any point.
As a result, the subpower Higson corona and the Higson corona of $X$
 are not locally connected at any point, respectively.
Using this, for the $n$-dimensional Euclidean space $\Bbb R^n$ with the usual metric and $n\geq2$,
 we conclude that the sublinear Higson corona of $\Bbb R^n$ is not locally connected at any point.
Moreover,  we prove that the sublinear Higson corona of $\Bbb R^n$ is mutually aposyndetic.

\section{Preliminaries}
Let $X$ be a completely regular space and $V$ an open subset of $X$.
For a compactification $\alpha X$ of $X$, we define
$$\Ext{\alpha X}V=\alpha X\pois \Cl{\alpha X}(X\pois V),$$
which is the largest open subset of $\alpha X$ whose intersection with $X$ is equal to $V$.
If there is no confusion, we denote $\tilde{V}$ instead of $\Ext{\alpha X}V$.

A compactification $\alpha X$ of a completely regular space $X$ is said to be {\it perfect}
if the natural projection from the Stone-\v{C}ech compactification $\beta X$ to $\alpha X$ is monotone.
Let $F$ be a subset of $\alpha X$.
We say that $F$ {\it separates} $\alpha X$ {\it at the point} $a\in F$
 if $a$ has a neighborhood $W$ such that $W\cap(\alpha X\pois F)=U\cup V$,
 where $U$ and $V$ are non empty open sets of $\alpha X\pois F$,
 $U\cap V=\emptyset$ and $a\in\Cl{\alpha X}U\cap\Cl{\alpha X}V$. 
The following basic facts of perfect compactifications are known (see \cite{KN} or \cite{Sk}).
\begin{prop}\label{prop:perfectcpt}
Let $\alpha X$ be a compactification of a noncompact completely regular space $X$.
Then the following conditions are equivalent$:$
\begin{itemize}
\item[$($1$)$] $\alpha X$ is a perfect compactification of $X$.
\item[$($2$)$] $\alpha X\pois X$ does not separate $\alpha X$ at any of its points.
\item[$($3$)$] For every open set $U$ of $X$, $\Cl{\alpha X}\Fr{X}U=\Fr{\alpha X}\Ext{\alpha X}U$.
\end{itemize}
\end{prop}

A (not necessarily continuous) function $s:\Bbb R_+\to\Bbb R_+$
 between the set of positive real numbers
is called {\it asymptotically sublinear} (resp.\ {\it asymptotically subpower})
if for every $\alpha>0$ there exists $t_0>0$ such that $s(t)<\alpha t$  (resp.\ $s(t)<t^{\alpha}$) for all $t>t_0$.
The set of all asymptotically sublinear functions (resp.\ all asymptotically subpower functions, \ 
all positive constant functions)
 is denoted by $\mathcal{L}$ (resp.\ $\mathcal{P}$, \ $\mathcal{H}$).
Observe that $\mathcal{H}\subsetneqq\mathcal{P}\subsetneqq\mathcal{L}$.
Indeed, $s(t)=\ln(1+t)$ is an asymptotically subpower function, and 
$s(t)=\sqrt{t}$ is an asymptotically sublinear function which is not an asymptotically subpower function.
Furthermore, let $\mathcal{A}$ be either $\mathcal{P}$ or $\mathcal{L}$. 
We can easily verify that  $\mathcal{A}$ is closed under addition and positive scalar multiple.

Let $(X,d)$ and $(Y,\rho)$ be proper metric spaces.
A map $f:X\to Y$ is called {\it Higson sublinear} (resp.\ {\it Higson subpower, Higson})
 provided that
\[\lim_{|x|\to\infty}\Diam{\rho}f(B_d(x,s(|x|)))=0\quad\quad(\ast)_s\]
for each $s\in\mathcal{L}$ (resp.\ $s\in\mathcal{P}$, \ $s\in\mathcal{H}$),
 i.e., for every $s\in\mathcal{L}$ (resp.\ $s\in\mathcal{P}$, $s\in\mathcal{H}$) and $\varepsilon>0$ there exists $r_0>0$ such that $\Diam{\rho}f(B_d(x,s(|x|)))<\varepsilon$
 for all $x\in X$ with $|x|>r_0$.

Let $C^{\ast}(X)$ be the set of all bounded real valued continuous functions on $X$.
For a noncompact proper metric space $(X,d)$ and
 for each $(K,\mathcal{K})\in\{(L,\mathcal{L}),(P,\mathcal{P}),(H,\mathcal{H})\}$,
 we consider the following subsets of $C^{\ast}(X)$:
$$C_K(X)=\{f\in C^{\ast}(X)~|~f~\mbox{satisfies}\ (\ast)_s\ \mbox{for every}\ s\in\mathcal{K}\}.$$
Here, each $C_K(X)$ is a closed subring of $C^{\ast }(X)$
containing the constants and generate the topology of $X$. 
The {\it sublinear Higson compactification} 
 (resp.\ {\it subpower Higson compactification}, {\it Higson compactification}) 
 of a noncompact proper metric space $(X,d)$
 is the unique compactification associated with the closed subring
 $C_L(X)$ (resp.\ $C_P(X)$,  $C_H(X)$),
 denoted by $h_L X$ (resp.\ $h_P X$,  $h_H X$).
Note that these compactifications are metric-dependent and
$h_L X\leq h_P X\leq h_H X$ since $C_L(X)\subseteq C_P(X)\subseteq C_H(X)$.
The remainder $h_L X\setminus X$  (resp.\ $h_P X\setminus X$, $h_H X\setminus X$) 
of the compactification $h_L X$ (resp.\ $h_P X$,  $h_H X$) is called the
{\it sublinear Higson corona}  (resp.\  {\it subpower Higson corona},  {\it Higson corona}) and is denoted by  $\nu_L X$ 
(resp.\ $\nu_P X$,  $\nu_H X$).
Here, it is well-known that these compactifications are characterized as follows:

\begin{itemize}
\item[$(\natural)$] A bounded map $f:X\to\Bbb R$ has an extension $\bar{f}:h_K X\to\Bbb R$
 if and only if $f\in C_K(X)$,
\end{itemize}
where $K\in\{L,P,H\}$.

We call a finite system $\{E_1,~E_2,~\ldots,~E_n\}$ of closed subsets in $(X,d)$ {\it diverges as a linear function}
 (see \cite{DS}) if there exist $\alpha>0$ and $r_0>0$ such that 
 $\mathrm{max}\{d(x,E_i)~|~1\leq i\leq n\}\geq\alpha|x|$ whenever $|x|>r_0$.

\begin{prop}[{\cite[Lemma 2.3]{DS}}]\label{prop:l_div} 
Let $($$X,d$$)$ be a noncompact proper metric space.
For a system $\{E_1,~E_2,~\ldots,~E_n\}$ of closed subsets of $X$, the following conditions are equivalent$:$
\begin{itemize}
\item[$($1$)$] $\nu_LX\cap(\cap_{i=1}^n\Cl{h_L X}E_i)=\emptyset$,
\item[$($2$)$] $\{E_1,~E_2,~\ldots,~E_n\}$ diverges as a linear function.
\end{itemize}
\end{prop}

In the case of $n=2$, we can formulate the fact above as follows.

\begin{lem}[{\cite[Lemma 2.4]{DS}}]\label{lem:r_subl_div} 
Let $A$ and $B$ be subsets of a metric space $X$.
Then the following are equivalent$:$
\begin{itemize}
\item[$($1$)$] There exist $C$, $r_0>0$ such that $\max\{d(x,A),d(x,B)\}\geq C|x|$ whenever $|x|\geq r_0$,
\item[$($2$)$] there exist $D$, $r_1>0$ such that 
$d(A\setminus B_d(x_0,r),B\setminus B_d(x_0,r))\geq Dr$ whenever $r\geq r_1$.
\end{itemize}
\end{lem}

\begin{df}
{\rm
Let $(X,d)$ be a noncompact proper metric space and let $\mathcal{V}$ be a family of open subsets of $X$.
The {\it Lebesugue function} associated with the cover $\mathcal{V}$ is defined by
$$L^{\mathcal{V}}(x)=\sup_{V\in\mathcal{V}}d(x,X\pois V).$$
}
\end{df}

\begin{df}
{\rm
For a proper metric space $(X,d)$, we say a function $f:X\to[0,\infty)$ is ({\it eventually}) 
 {\it at least linear} if there exist $c$, $r_0>0$ such that $f(x)\geq c|x|$ whenever $|x|\geq r_0$.
}
\end{df}

\begin{prop}[{\cite[Corollary 2.5]{DS}}]\label{prop:cover}
Let $(X,d)$ be a noncompact proper metric space and
 let $\mathcal{V}=\{V_1,V_2,\ldots,V_n\}$ be a finite family of open subsets of $X$.
Then $\tilde{\mathcal{V}}=\{\tilde{V}_1,\tilde{V}_2,\ldots,\tilde{V}_n\}$ covers the corona $\nu_LX$
 if and only if the Lebesgue function $L^{\mathcal{V}}$ is at least linear.
\end{prop}

\begin{lem}\label{lem:cap_cld}
Let $(X,d)$ be a proper metric space, $0<\alpha<1/5$, $x$, $y\in X$ 
 and $N_d(x,\alpha|x|)\cap N_d(y,\alpha|y|)\ne\emptyset$.
Then $N_d(y,\alpha|y|)\subset B_d(x,4\alpha|x|)$.
\end{lem}
\begin{proof}
If $0<\alpha<1/5$ then $(3\alpha+\alpha^2)/(1-\alpha)<4\alpha$ since
 $4\alpha-(3\alpha+\alpha^2)/(1-\alpha)=\alpha(1-5\alpha)/(1-\alpha)>0$.
By assumption, 
 we have $d(x,y)\leq\alpha|x|+\alpha|y|$.
If $|y|>|x|+\alpha|x|+\alpha|y|$, then
$$|y|>|x|+d(x,y)=d(x,x_0)+d(x,y)\geq d(y,x_0)=|y|,$$
 a contradiction. 
Hence $|y|\leq|x|+\alpha|x|+\alpha|y|$, i.e., $|y|\leq(1+\alpha)|x|/(1-\alpha)$.
For any $w\in N_d(y,\alpha|y|)$, 
\begin{align*}
d(w,x) & \leq d(w,y)+d(y,x)\leq\alpha|y|+\alpha|x|+\alpha|y|=\alpha|x|+2\alpha|y| \\
 & \leq\alpha|x|+2\alpha\cdot\frac{1+\alpha}{1-\alpha}|x|=\frac{3\alpha+\alpha^2}{1-\alpha}|x|<4\alpha|x|,
\end{align*}
hence $N_d(y,\alpha|y|)\subset B_d(x,4\alpha|x|)$.
\end{proof}

Let $\Bbb N=\{1,2,3,\ldots\}$ be the set of natural numbers,
 and let $\mathcal{U}$ be a collection of subsets of $X$ and $Y\subset X$.
We define $\St{Y}{U}=\bigcup\{U\in\mathcal{U}~|~U\cap Y\ne\emptyset\}$.

\begin{lem}\label{lem:div}
Let $(X,d)$ be a noncompact proper metric space, $A$ and $B$ be closed sets in $X$.
If there exist $\alpha$ with $0<\alpha<1/5$ and $r_0>0$ such that 
 $\max\{d(x,A),d(x,B)\}\geq4\alpha|x|$ whenever $|x|\geq r_0$,
 then $$\max\{d(x,\St{A}{U}),d(x,\St{B}{U})\}\geq\alpha|x|$$ whenever $|x|\geq r_0$,
 where $\mathcal{U}=\{N_d(x,\alpha|x|)~|~x\in X\}$.
\end{lem}
\begin{proof}
Suppose not.
Then there exist sequences $\{x_i\}\subset X$,
 $\{u_i\}_{i\in\Bbb N}\subset \St{A}{U}$,
 $\{v_i\}_{i\in\Bbb N}\subset \St{B}{U}$
 such that $\lim_{i\to\infty}|x_i|=\infty$,
 $d(x_i,u_i)<\alpha|x_i|$, $d(x_i,v_i)<\alpha|x_i|$.
Moreover, there exist sequences $\{p_i\}_{i\in\Bbb N}$, $\{q_i\}_{i\in\Bbb N}\subset X$ 
 such that $u_i\in N_d(p_i,\alpha|p_i|)$, $A\cap N_d(p_i,\alpha|p_i|)\ne\emptyset$ and
 $v_i\in N_d(q_i,\alpha|q_i|)$, $B\cap N_d(q_i,\alpha|q_i|)\ne\emptyset$.
Since $u_i\in N_d(x_i,\alpha|x_i|)\cap N_d(p_i,\alpha|p_i|)$ and
 $v_i\in N_d(x_i,\alpha|x_i|)\cap N_d(q_i,\alpha|q_i|)$,
 by Lemma \ref{lem:cap_cld},
 $N_d(p_i,\alpha|p_i|)\cup N_d(q_i,\alpha|q_i|)\subset B_d(x_i,4\alpha|x_i|)$,
 which implies $\max\{d(x_i,A),d(x_i,B)\}<4\alpha|x_i|$,
 a contradiction.
\end{proof}

\section{Non local connectedness and mutual aposyndesis of $\nu_L\Bbb R^n$}\label{mut-ap.chap}

A topological space $X$ is said to be {\it locally connected at}
 $p\in X$ provided that each neighborhood of $p$ contains a connected open neighborhood of $p$.
A topological space $X$ is said to be {\it connected im kleinen at}
 $p\in X$ provided that each neighborhood of $p$ contains a connected neighborhood of $p$.
It is known that $X$ is locally connected at every point if and only if
 $X$ is connected im kleinen at every point (cf.\ \cite[1.4.18]{Ma}).

\begin{prop}\label{prop:notLC} 
Let $($$X,d$$)$ be a noncompact proper metric space.
Then $\nu_LX$ is not locally connected at any point.
\end{prop}
\begin{proof}
Let $\rho$ be the usual metric on the real number and fix a real number $c$ with $0<c<1/2$.
Define a sequence $\{a_n\}_{n\in\Bbb N}$ by
$a_{n+1}=1/(1-2c)^n$ for each $n\in\{0\}\cup\Bbb N$.
Since $0<1-2c<1$, we have $\lim_{n\to\infty}a_n=\infty$.
Note that
$$a_{n+1}-a_n=\frac{a_n}{1-2c}-a_n=\frac{2ca_n}{1-2c}\quad \mbox{and}\quad
ca_{n+1}=c\cdot\frac{a_n}{1-2c}=\frac{1}{2}\cdot\frac{2ca_n}{1-2c}=\frac{a_{n+1}-a_n}{2}.$$
Let $s$ be an asymptotically sublinear function.
Then there exists $N\geq2$ such that $s(t)<2ct$ for any $n\geq N$ and any $t\in [a_n,a_{n+1}]$.
Let $n\geq N$ and $t\in[a_n,a_{n+1}]$. 
Since
$$a_{n-1}=(1-2c)a_n\leq (1-2c)t=t-2ct$$ and
$$t+2ct\leq a_{n+1}+2ca_{n+1}=(1+2c)a_{n+1}=(1+2c)(1-2c)a_{n+2}<a_{n+2},$$
 we have 
\begin{itemize}
\item[$(\dagger)$]
 $B_{\rho}(t,s(t))\subset B_{\rho}(t,2ct)\subset[a_{n-1},a_{n+2}]$.
\end{itemize}
Since $t\in[a_n,a_{n+1}]$, we have $ca_n\leq ct\leq ca_{n+1}$.
If $t-a_n\leq(a_{n+1}-a_n)/2$, then
$$a_{n+1}-t\geq\frac{a_{n+1}-a_n}{2}=ca_{n+1}\geq ct.$$
If $t-a_n\geq(a_{n+1}-a_n)/2$, then $t-a_n\geq ca_{n+1}\geq ct$.
Hence we claim that
\begin{itemize}
\item[$(\ast)$] if $t\in[a_n,a_{n+1}]$, then $\max\{t-a_n,a_{n+1}-t\}\geq ct$.
\end{itemize}

Define two maps $f^+_n$, $f^-_n:[0,\infty)\to[0,1]$ as
$$f^+_n(t)=
\left\{
\begin{array}{cl}
0                     & \mbox{if}\ 0\leq t\leq a_n,\\
\displaystyle\frac{t-a_n}{a_{n+1}-a_n} & \mbox{if}\ a_n\leq t\leq a_{n+1},\\
1                     & \mbox{if}\ t\geq a_{n+1}
\end{array}
\right.$$
and
$$f^-_n(t)=
\left\{
\begin{array}{cl}
1                         & \mbox{if}\ 0\leq t\leq a_n,\\
\displaystyle\frac{a_{n+1}-t}{a_{n+1}-a_n} & \mbox{if}\ a_n\leq t\leq a_{n+1},\\
0                         & \mbox{if}\ t\geq a_{n+1}.
\end{array}
\right.$$
The slope of $f_n$ on $[a_n,a_{n+1}]$ is $1/(a_{n+1}-a_n)=(1-2c)^n/(2c)$
 and the slope of $f^-_n$ on $[a_n,a_{n+1}]$ is $-(1-2c)^n/(2c)$.
Let $s$ be an asymptotically sublinear function.
Then for any $\varepsilon>0$ there exists $r_0>0$ such that
 $s(t)<\varepsilon ct$ whenever $t\geq r_0$.
If $t\geq r_0$ and $t\in[a_n,a_{n+1}]$, then
\begin{align*}
 (\ddagger)\qquad \Diam{\rho}f^+_n(B_{\rho}(t,s(t)))&\leq\frac{(1-2c)^n}{2c}\cdot2s(t)<\frac{(1-2c)^n}{2c}\cdot2\varepsilon ct\\
 &\leq\varepsilon(1-2c)^nt\leq\varepsilon(1-2c)^n\cdot\frac{1}{(1-2c)^n}=\varepsilon
\end{align*}
since $t\leq a_{n+1}=1/(1-2c)^n$.
This inequality also holds for $f^-_n$.

Let $(X,d)$ be a noncompact proper metric space.
The following technique, dividing into four parts, is stated in \cite{Chi}.
Define $K_n=\{x\in X~|~a_n\leq|x|\leq a_{n+1}\}$ and
$$X_0=\bigcup_{i\in\Bbb N}K_{4i},\ X_1=\bigcup_{i\in\Bbb N}K_{4i+1},
\ X_2=\bigcup_{i\in\Bbb N}K_{4i+2}\ \mbox{and}\  X_3=\bigcup_{i\in\Bbb N}K_{4i+3}.$$
If $x\in K_n$, then $|x|\in[a_n,a_{n+1}]$.
Since $d(x,K_{n-1})+a_n\geq|x|$ and $|x|+d(x,K_{n+1})\geq a_{n+1}$,
 we have $\max\{d(x,K_{n-1}),d(x,K_{n+1})\geq\max\{|x|-a_n,a_{n+1}-|x|\}\geq c|x|$ by $(\ast)$.
This implies that $\{X_0,X_2\}$ and $\{X_1,X_3\}$ diverge as a linear function, respectively.

For any $z\in\nu_LX$, since $\nu_LX=\nu_L(X_0\cup X_1\cup X_2\cup X_3)=\bigcup_{i=0}^3(\Cl{h_LX}X_i\pois X_i)$,
 there exists $i=0,1,2,3$ such that $z\in\Cl{h_LX}X_i\pois X_i$.
Here, we assume $z\in\Cl{h_LX}X_2\pois X_2$,
 and other cases are valid in the following argument.

Since $\{X_0,X_2\}$ diverges as a linear function,
 by Proposition \ref{prop:l_div}, 
 we have $z\notin\Cl{h_LX}X_0\pois X_0$.
Hence $\Cl{h_LX}(X_1\cup X_2\cup X_3)\pois(X_1\cup X_2\cup X_3)$ is a closed neighborhood of $z$ in $\nu_LX$.
Set $W_{2i}=K_{4\cdot2i+1}\cup K_{4\cdot2i+2}\cup K_{4\cdot2i+3}$ and
 $W_{2i+1}=K_{4(2i+1)+1}\cup K_{4(2i+1)+2}\cup K_{4(2i+1)+3}$ for each $i\in\Bbb N$.
Let $V$ be an open neighborhood of $z$ in $\nu_LX$ such that
 $V\subset\Cl{h_LX}(X_1\cup X_2\cup X_3)\pois(X_1\cup X_2\cup X_3)$,
 and let $V'$ be a closed neighborhood of $z$ in $h_LX$ such that
 $V'\cap\nu_LX\subset V$.
Since $\nu_LX\cap V'\cap\Cl{h_LX}X_0=\emptyset$,
 by Proposition \ref{prop:l_div},
 $\{V'\cap X,X_0\}$ diverges as a linear function,
 so we may assume that $V'\cap X\subset X_1\cup X_2\cup X_3$.
Put $J=\{j\in\Bbb N~|~V'\cap W_j\ne\emptyset\}$.
Since $z\in V'\cap\nu_LX$, the cardinal of $J$ is infinite.
Let $\phi:\Bbb N\to J$ be an order preserving bijection.
Note that $V'\cap X\subset\bigcup_{j\in\Bbb N}W_{\phi(j)}$.
Define a map $g:X\to[0,1]$ as follows: for each $j\in\Bbb N$,
$$g(x)=
\left\{
\begin{array}{cl}
f^+_{4\phi(2j)}(|x|) & \mbox{if}\ x\in K_{4\phi(2j)},\\
1                        & \mbox{if}\ x\in W_{\phi(2j)},\\
f^-_{4(\phi(2j)+1)}(|x|) & \mbox{if}\ x\in K_{4(\phi(2j)+1)},\\
0                        & \mbox{otherwise.}
\end{array}
\right.$$
Note that $g(x)=0$ if $x\in W_{\phi(2j+1)}$.
By $(\dagger)$ and $(\ddagger)$, $g$ is a Higson sublinear map
 because for any $x\in X$ and $y\in B_d(x,s(|x|))$ we have $|y|\in B_{\rho}(|x|,s(|x|))$.
By $(\natural)$, there exists an extension $\bar{g}:h_LX\to[0,1]$ of $g$.
For each $j\in\Bbb N$, take $c_{2j}\in V'\cap W_{\phi(2j)}$ and $c_{2j+1}\in V'\cap W_{\phi(2j+1)}$
 and let $p$ and $q$ be cluster points of $\{c_{2j}\}_{j\in\Bbb N}$ and $\{c_{2j+1}\}_{j\in\Bbb N}$, respectively.
Since $g(\{c_{2j}\})=1$ and $g(\{c_{2j+1}\})=0$,
 we have 
  $\bar{g}(p)=1$ and $\bar{g}(q)=0$ by the continuity of $\bar{g}$.
Hence $\{0,1\}\subset\bar{g}(V'\cap\nu_LX)$.
Assume that there exists a point $z'\in V'\cap\nu_LX$ such that $\bar{g}(z')\in(0,1)$.
Then there exists a neighborhood $W$ of $z'$ in $h_LX$ such that $\bar{g}(W)\subset(0,1)$,
 so $\bar{g}(W\cap V'\cap X)=g(W\cap V'\cap X)\subset(0,1)$.
But by the definition of $g$, $\bar{g}(V'\cap X)=g(V'\cap X)\subset\{0,1\}$, a contradiction.
Hence $\bar{g}(V'\cap\nu_LX)=\{0,1\}$, that is, $V'\cap\nu_LX$ is not connected.
Therefore, $\nu_LX$ is not connected im kleinen at any point,
 so $\nu_LX$ is not locally connected at any point.
\end{proof}

\begin{rem}{\rm 
It is known that the Higson corona of any connected noncompact proper metric space with some conditions
 is not locally connected at any point $($cf.\ \cite[Theorem 5]{Kees}$)$.
Let $($$X,d$$)$ be a noncompact proper metric space.
Since $C_L(X)\subseteq C_P(X)\subseteq C_H(X)$ and using $(\natural)$,
 the map $g$ in the proof of Proposition $\ref{prop:notLC}$ has extensions
 $g_1:h_PX\to[0,1]$ and $g_2:h_HX\to[0,1]$, respectively.
Hence, by the same argument of the proof of Proposition $\ref{prop:notLC}$,
 we can conclude that $\nu_PX$ and $\nu_HX$ are also not locally connected at any point.
}
\end{rem}

A compact connected Hausdorff space is said to be {\it continuum}.
A {\it subcontinuum} is a continuum contained in a space.
Let $X$ be a continuum, $W$ be a subcontinuum of $X$ and $p\in W$.
We say $W$ a {\it continuum neighborhood} of $p$ if $p\in\Int{X}W$.
The continuum $X$ is {\it aposyndetic}
 provided that every two distinct points $x$ and $y$ of $X$
 there exists a continuum neighborhood $W$ of $x$ such that $y\notin W$;
 $X$ is {\it mutually aposyndetic}
 provided that every two distinct points of $X$ have disjoint continuum neighborhoods.
See \cite{Ma} for more details.

Now consider the $n$-dimensional Euclidean space $\Bbb R^n$ with the usual metric $d$.

\begin{thm}\label{thm:m-apo}
For every $n\in\Bbb N$ with $n\geq2$, $\nu_L\Bbb R^n$ is a mutually aposyndetic continuum.
\end{thm}
\begin{proof}
Let the origin $O$ of $\Bbb R^n$ be a base point. 
We denote a point $x\in\Bbb R^n$ by the polar coordinates $x=(r,\theta)$,
 where $r$ is the distance to the origin, i.e., $r=|x|$,
 and $\theta$ represents a point on the unit sphere $S^{n-1}$ centered at $O$.
For brevity, let $X=\Bbb R^n$.
 
For any small $\alpha>0$ and $\theta\in S^{n-1}$, let
$l_{\theta}=\{(r,\theta)~|~r\geq1\}$,
$B_{\theta}^{\alpha}=B_d(\theta,\alpha)=\Int{X}N_d(\theta,\alpha)$, and
$$C_{\theta}^{\alpha}=\bigcup\{B_d(x,\alpha|x|)~|~x\in l_{\theta}\}
 =\bigcup\{B_d((r,\theta),\alpha r)~|~r\geq1\},$$
where $x=(r,\theta)$ and $(1,\theta)=\theta$.
Note that $\Cl{X}C_{\theta}^{\alpha}=\bigcup\{N_d(x,\alpha|x|)~|~x\in l_{\theta}\}$
 and $\Cl{h_LX}C_{\theta}^{\alpha}\cap\nu_LX
 =\bigcap\{ \Cl{h_LX}C_{\theta}^{\alpha}\pois B_d(O,m)~|~m\in\Bbb N\}$
 is a subcontinuum of $\nu_LX$.

{\bf Fact}. $\Cl{h_LX}\tilde{C}_{\theta}^{\alpha}=\Cl{h_LX}C_{\theta}^{\alpha}$.

Indeed, $\Cl{h_LX}\tilde{C}_{\theta}^{\alpha}\supset\Cl{h_LX}C_{\theta}^{\alpha}$
 since $\tilde{C}_{\theta}^{\alpha}\supset C_{\theta}^{\alpha}$.
For any $x\in\Cl{h_LX}\tilde{C}_{\theta}^{\alpha}$ and any neighborhood $V$ of $x$ in $h_LX$,
 we have $V\cap C_{\theta}^{\alpha}=V\cap\tilde{C}_{\theta}^{\alpha}\cap X\ne\emptyset$
 because $X$ is dense in $h_LX$.
So $x\in\Cl{h_LX}C_{\theta}^{\alpha}$, which implies $\Cl{h_LX}\tilde{C}_{\theta}^{\alpha}\subset\Cl{h_LX}C_{\theta}^{\alpha}$.

{\bf Fact}. $\Cl{h_LX}\Cl{X}C_{\theta}^{\alpha}=\Cl{h_LX}\tilde{C}_{\theta}^{\alpha}$.

Indeed, by \cite{AT},
 $h_LX$ is a perfect compactification, and by Proposition \ref{prop:perfectcpt} (3),
$\Cl{h_LX}\Fr{X}C_{\theta}^{\alpha}=\Fr{h_LX}\tilde{C}_{\theta}^{\alpha}$.
Then
\begin{align*}
\Cl{h_LX}\Cl{X}C_{\theta}^{\alpha}
&=\Cl{h_LX}(C_{\theta}^{\alpha}\cup\Fr{X}C_{\theta}^{\alpha})
=\Cl{h_LX}C_{\theta}^{\alpha}\cup\Cl{h_LX}\Fr{X}C_{\theta}^{\alpha}\\
&=\Cl{h_LX}\tilde{C}_{\theta}^{\alpha}\cup\Fr{h_LX}\tilde{C}_{\theta}^{\alpha}
=\Cl{h_LX}\tilde{C}_{\theta}^{\alpha}. 
\end{align*}

Since $S^{n-1}$ is a compact,
 there exists a finite subset $\{\theta_l\}_{l=1}^k$ in $S^{n-1}$ such that
 $\{B_{\theta_l}^{\alpha}\}_{l=1}^k$ is an open cover of $S^{n-1}$ in $X$.
Consider a finite family $\mathcal{C}=\{C_{\theta_l}^{\alpha}\}_{l=1}^k$.

{\bf Fact}. The Lebesgue function $L^{\mathcal{C}}$ is at least linear.

Indeed, let $x=(r,p)\in X$ be any point.
By \cite[Lemma 1.1.1]{Mill},
 there exists $\eta>0$ such that
 $B_d(p,\eta)$ is contained in an element of $\{B_{\theta_l}^{\alpha}\}_{l=1}^k$.
Let $B_d(p,\eta)\subset B_{\theta_{l'}}^{\alpha}$ for an $l'$ with $1\leq l'\leq k$. 
Then for any $r\geq1$ we have $B_d((r,p),\eta r)\subset B_d((r,\theta_{l'}),\alpha r)$ by dilatation,
 so $B_d((r,p),\eta r)\subset C_{\theta_{l'}}^{\alpha}$.
Hence $B_d(x,\eta |x|)\subset C_{\theta_{l'}}^{\alpha}$
 which implies that $L^{\mathcal{C}}(x)\geq \eta|x|$, thus $L^{\mathcal{C}}$ is at least linear.

By Proposition \ref{prop:cover},
 $\tilde{\mathcal{C}}=\{\tilde{C}_{\theta_l}^{\alpha}\}_{l=1}^k$ covers $\nu_LX$.
Hence, for any $z\in\nu_LX$, there exists $l$ with $1\leq l\leq k$ such that $z\in\tilde{C}_{\theta_l}^{\alpha}$ and
 $W_{\theta_l}=\Cl{h_LX}\tilde{C}_{\theta_l}^{\alpha}\cap\nu_LX$ is a continuum neighborhood of $z$ in $\nu_LX$.

Let $z,~w\in\nu_LX$ be distinct points.
Consider the following statement:
\begin{itemize}
\item[$(\ast)_{(z,w)}$] there exist $\alpha>0$ and a finite subset $\{\theta_l\}_{l=1}^k$ in $S^{n-1}$
 such that $z\in W_{\theta_i}$, $w\in W_{\theta_j}$, $W_{\theta_i}\cap W_{\theta_j}=\emptyset$,
 where $1\leq i\leq k$, $1\leq j\leq k$, $i\ne j$.
\end{itemize}
If $(\ast)_{(z,w)}$ holds, then we can conclude that $\nu_LX$ is a mutually aposyndetic continuum.

If $(\ast)_{(z,w)}$ does not hold,
 then for any $\alpha>0$, any $\{\theta_i\}_{l=1}^k$ in $S^{n-1}$,
 we have $W_{\theta_i}\cap W_{\theta_j}\ne\emptyset$
 whenever $z\in W_{\theta_i}$ and $w\in W_{\theta_j}$.
Since 
 $\Cl{h_LX}\tilde{C}_{\theta_i}^{\alpha}\cap X=\Cl{X}C_{\theta_i}^{\alpha}$ and 
 $\Cl{h_LX}\tilde{C}_{\theta_j}^{\alpha}\cap X=\Cl{X}C_{\theta_j}^{\alpha}$,
 by Proposition \ref{prop:l_div},
 $\{\Cl{X}C_{\theta_i}^{\alpha},\Cl{X}C_{\theta_j}^{\alpha}\}$
 does not diverge as a linear function.
If $\Cl{X}C_{\theta_i}^{\alpha}\cap\Cl{X}C_{\theta_j}^{\alpha}=\emptyset$,
 then $\varepsilon=d(B_{\theta_i}^{\alpha},B_{\theta_j}^{\alpha})>0$,
 so, by dilatation,
  $d(\Cl{X}C_{\theta_i}^{\alpha}\pois B_d(O,r),\Cl{X}C_{\theta_j}^{\alpha}\pois B_d(O,r))\geq\varepsilon r$
 whenever $r\geq1$.
By Lemma \ref{lem:r_subl_div}, $\{\Cl{X}C_{\theta_i}^{\alpha},\Cl{X}C_{\theta_j}^{\alpha}\}$
 diverges as a linear function, a contradiction.
Thus $\Cl{X}C_{\theta_i}^{\alpha}\cap\Cl{X}C_{\theta_j}^{\alpha}\ne\emptyset$.
Then there exists $p\in\Cl{X}B_{\theta_i}^{\alpha}\cap\Cl{X}B_{\theta_j}^{\alpha}$ such that 
 $\Cl{X}C_{\theta_i}^{\alpha}\cup\Cl{X}C_{\theta_j}^{\alpha}
 \subset\Cl{X}C_{p}^{2\alpha}$,
 so, we have
 $\Cl{h_LX}\Cl{X}C_{\theta_i}^{\alpha}\cup\Cl{h_LX}\Cl{X}C_{\theta_j}^{\alpha}
  \subset\Cl{h_LX}\Cl{X}C_{p}^{2\alpha}$, 
  hence $\Cl{h_LX}\tilde{C}_{\theta_i}^{\alpha}\cup\Cl{h_LX}\tilde{C}_{\theta_j}^{\alpha}
  \subset\Cl{h_LX}\tilde{C}_{p}^{2\alpha}$. 
Hence $W_{\theta_i}\cup W_{\theta_j}\subset\Cl{h_LX}\tilde{C}_{p}^{2\alpha}\cap\nu_LX$.
This implies that,
 if $(\ast)_{(z,w)}$ does not hold, then for any $\alpha>0$ there exists $p\in S^{n-1}$
 such that $z,~w\in\Cl{h_LX}\tilde{C}_{p}^{\alpha}=\Cl{h_LX}\Cl{X}C_{p}^{\alpha}$
 and $\Cl{h_LX}\tilde{C}_{p}^{\alpha}$ is a closed neighborhood of $z$ and $w$.

From now on, assume that $(\ast)_{(z,w)}$ does not hold.
Since $h_LX$ is compact Hausdorff,
 there exist closed neighborhoods $U$ and $V$ of $z$ and $w$,
 respectively, in $h_LX$ such that $U\cap V=\emptyset$.
By Proposition \ref{prop:l_div},
 $U\cap X$ and $V\cap X$ diverges as a linear function,
 hence there exist $\alpha>0$ and $r_0>0$ such that
 $\max\{d(x,U\cap X),d(x,V\cap X)\}\geq4\alpha|x|$ whenever $|x|\geq r_0$.
We may assume $\alpha<1/5$.
Moreover, there exists $p\in S^{n-1}$ such that
 $z,~w\in\tilde{C}_{p}^{\alpha}\subset\Cl{h_LX}\tilde{C}_{p}^{\alpha}=\Cl{h_LX}\Cl{X}C_{p}^{\alpha}$.
Note that $z\in U\cap\Cl{h_LX}\Cl{X}C_{p}^{\alpha}$ and $w\in V\cap\Cl{h_LX}\Cl{X}C_{p}^{\alpha}$.
Let $\mathcal{W}=\{N_d(x,\alpha|x|)~|~x\in X\}$.
By Lemma \ref{lem:div},
$$\max\{d(x,\St{U\cap X}{W}),d(x,\St{V\cap X}{W})\}\geq\alpha|x|$$
 whenever $|x|\geq r_0$.
Set $A=\bigcup\{N_d(x,\alpha|x|)~|~x\in l_p,\ N_d(x,\alpha|x|)\cap U\ne\emptyset\}$ and
 $B=\bigcup\{N_d(x,\alpha|x|)~|~x\in l_p,\ N_d(x,\alpha|x|)\cap V\ne\emptyset\}$.
Since $z\in\Int{h_LX}U\cap\tilde{C}_{p}^{\alpha}$ and $w\in\Int{h_LX}V\cap\tilde{C}_{p}^{\alpha}$,
 $A$ and $B$ are not empty and $\sup\{|a|~|~a\in A\}=\sup\{|b|~|~b\in B\}=\infty$ because $X$ is dense in $h_LX$.
Note that $A\subset\St{U\cap X}{W}$ and $B\subset\St{V\cap X}{W}$,
 so $\max\{d(x,A),d(x,B)\}\geq\alpha|x|$ whenever $|x|\geq r_0$.
Moreover, $U\cap\Cl{X}C_{p}^{\alpha}\subset A$ and $V\cap\Cl{X}C_{p}^{\alpha}\subset B$, 
 hence $\Cl{h_LX}A$ and $\Cl{h_LX}B$ are closed neighborhoods
  of $z$ and $w$, respectively.

For any subset $Y$ of $X$, we denote that
 $r_Y=\inf\{|y|~|~y\in Y\}$, $r^Y=\sup\{|y|~|~y\in Y\}$ and
 $T(Y)=\{(r,\theta)~|~r_Y\leq r\leq r^Y,\ \theta\in S^{n-1}\}$.
Let $\{A_i~|~i\in\Bbb N\}$ and $\{B_i~|~i\in\Bbb N\}$ 
 be collections of components of $A$ and $B$, respectively.
We may assume that $r^{A_i}<r_{A_j}$ and $r^{B_i}<r_{B_j}$ if $i<j$.
Define $A'=\bigcup_{i\in\Bbb N}T(A_i)$ and $B'=\bigcup_{i\in\Bbb N}T(B_i)$. 
Observe that $\max\{d(x,A'),d(x,B')\}\geq\alpha|x|$ whenever $|x|\geq r_0$.
Indeed, if there exist a sequence $\{x_i\}\subset X$ such that $\lim_{i\to\infty}|x_i|=\infty$ and
 $\max\{d(x_i,A'),d(x_i,B')\}<\alpha|x_i|$,
 by symmetry, for any $z_i=(|x_i|,p)$, we have $\max\{d(z_i,A),d(z_i,B)\}<\alpha|z_i|$, a contradiction.

Take points $p_A$, $p_B\in S^{n-1}$ such that each of
 $N_d(p_A,\alpha)\cap N_d(p,\alpha)$ and
 $N_d(p_B,\alpha)\cap N_d(p,\alpha)$ is one point, and
 $N_d(p_A,\alpha)\cap N_d(p_B,\alpha)=\emptyset$.
Then $\varepsilon=d(B_{p_A}^{\alpha},B_{p_B}^{\alpha})>0$.
By dilatation, 
 $d(\Cl{X}C_{p_A}^{\alpha}\pois B_d(O,r),\Cl{X}C_{p_B}^{\alpha}\pois B_d(O,r))\geq\varepsilon r$
 whenever $r\geq1$.
By Lemma \ref{lem:r_subl_div} and Proposition \ref{prop:l_div},
 $\{\Cl{X}C_{p_A}^{\alpha},\Cl{X}C_{p_B}^{\alpha}\}$ diverges as a linear function.
Similarly, $\{l_{p_A},\Cl{X}C_{p}^{\alpha}\}$, and $\{l_{p_B},\Cl{X}C_{p}^{\alpha}\}$
 diverge as a linear function, respectively.

Since $N_d(p_A,\alpha)\cap N_d(p,\alpha)$ is a point, say $a\in X$,
 we have $\Cl{X}C_{p_A}^{\alpha}\cap\Cl{X}C_{p}^{\alpha}\supset\{(r,a/|a|)~|~r\geq1\}$,
 where $a/|a|\in S^{n-1}$ as a unit vector.
Then there exists an arc $K_i$ in $\Cl{X}C_{p_A}^{\alpha}\cap T(A_i)$ connecting
 $l_{p_A}$ and $A_i$.
Similarly, there exists an arc $L_i$ in $\Cl{X}C_{p_B}^{\alpha}\cap T(B_i)$ connecting
 $l_{p_B}$ and $B_i$.
Let $M=l_{p_A}\cup A\cup\bigcup_{i\in\Bbb N}K_i$ and $N=l_{p_B}\cup B\cup\bigcup_{i\in\Bbb N}L_i$.
Note that $M$ and $N$ are connected, respectively,
 hence
 $\Cl{h_LX}M\cap\nu_LX$ and $\Cl{h_LX}N\cap\nu_LX$ are connected, respectively.
Since $\{l_{p_A},\Cl{X}C_{p}^{\alpha}\}$ diverges as a linear function
 and $B\subset\Cl{X}C_{p}^{\alpha}$,
 we have that $\{l_{p_A},B\}$ diverges as a linear function.
Since $\{\Cl{X}C_{p_A}^{\alpha},\Cl{X}C_{p_B}^{\alpha}\}$ diverges as a linear function,
 and $l_{p_A}\subset\Cl{X}C_{p_A}^{\alpha}$ and $l_{p_B}\cup\bigcup_{i\in\Bbb N}L_i\subset\Cl{X}C_{p_B}^{\alpha}$,
 we have that $\{l_{p_A},l_{p_B}\cup\bigcup_{i\in\Bbb N}L_i\}$ diverges as a linear function.
Since $\{A',B'\}$ diverges as a linear function,
 and $A\cup\bigcup_{i\in\Bbb N}K_i\subset A'$ and $B\cup\bigcup_{i\in\Bbb N}L_i\subset B'$,
 we have that $\{A\cup\bigcup_{i\in\Bbb N}K_i,B\cup\bigcup_{i\in\Bbb N}L_i\}$ diverges as a linear function.
Since $\{\Cl{X}C_{p_A}^{\alpha}\cup\Cl{X}C_{p}^{\alpha},l_{p_B}\}$ diverges as a linear function,
 and $A\cup\bigcup_{i\in\Bbb N}K_i\subset\Cl{X}C_{p_A}^{\alpha}\cup\Cl{X}C_{p}^{\alpha}$,
 we have that $\{A\cup\bigcup_{i\in\Bbb N}K_i,l_{p_B}\}$ diverges as a linear function.
Therefore, $\{M,N\}$ diverges as a linear function, 
 which implies $\Cl{h_LX}M\cap\Cl{h_LX}N=\emptyset$.
Since $\Cl{h_LX}M\cap\nu_LX$ and $\Cl{h_LX}N\cap\nu_LX$
 are continuum neighborhoods of $z$ and $w$, respectively,
 we conclude that $\nu_LX$ is a mutually aposyndetic continuum.
\end{proof}

By Proposition \ref{prop:notLC} and Theorem \ref{thm:m-apo}, we have the following.

\begin{cor}
For every $n\in\Bbb N$ with $n\geq2$,
$\nu_L\Bbb R^n$ is not locally connected at any point and is mutually aposyndetic.
\end{cor}

\begin{prop}
Let $Y=\Bbb R\times[0,\infty)\subset\Bbb R^2$ with a subspace usual metric.
Then $\nu_LY$ is not homeomorphic to $\nu_L\Bbb R^2$.
\end{prop}
\begin{proof}
We shall prove that $\nu_LY$ is not mutually aposyndetic.
For the Stone-\v{C}ech compactification,
 similar way is mentioned without proof in \cite[p.50]{DG}.
Let $L=\{(t,0)~|~t\geq0\}\subset Y$ and
 $x$, $y$ be distinct points in $\Cl{h_LY}L\pois L$.
If $\nu_LY$ is mutually aposyndetic,
 then there exist disjoint continuum neighborhoods $U$, $V$ of $x$, $y$, respectively, in $\nu_LY$.
Then there exist disjoint connected open neighborhoods
 $U'$ and $V'$ in $h_LY$ such that $U\subset U'$ and $V\subset V'$.
Since $x\in(\Cl{h_LY}L\pois L)\cap U'$ and $y\in(\Cl{h_LY}L\pois L)\cap V'$,
 there exist sequences $\{(x_i,0)\}_{i\in\Bbb N}\subset L\cap U'$
 and $\{(y_i,0)\}_{i\in\Bbb N}\subset L\cap V'$
 such that $\lim_{i\to\infty}x_i=\lim_{i\to\infty}y_i=\infty$.
There is no loss of generality that $x_n<y_n<x_{n+1}$ for any $n\in\Bbb N$.
By \cite{AT}, $h_LY$ is a perfect compactification of $Y$,
 and Proposition \ref{prop:perfectcpt},
$U'\cap Y$ is connected, hence $U'\cap Y$ is arcwise connected
 because $U'\cap Y$ is an open subset in $Y$.
Then there exists an arc $K_n$ connecting $(x_n,0)$ and $(x_{n+1},0)$ in $U'\cap Y$
 satisfying $K_n\cap L=\{(x_n,0),(x_{n+1},0)\}$.
Observe that $K_n'=K_n\cup\{(x,0)~|~x_n\leq x\leq x_{n+1}\}$ is a simple closed curve, and
 the bounded component of $Y\pois K_n'$ has a point which is arbitrarily close to $(y_n,0)$
 and $(y_{n+1},0)$ belongs to the unbounded component of $Y\pois K_n'$.
By the Jordan curve theorem,
 we have $K_n\cap V'\ne\emptyset$, that is, $U'\cap V'\ne\emptyset$.
This contradicts $U'\cap V'=\emptyset$.  
\end{proof}

\end{document}